\documentclass{article}
\usepackage{amsfonts,amsmath,amsthm,amssymb,latexsym,mathtools}
\usepackage{fullpage}
\usepackage[square,numbers]{natbib}
\usepackage{todonotes}
 
\usepackage{xcolor}
\usepackage[mathscr]{euscript}
\def\ds{\displaystyle}
\usepackage[utf8]{inputenc}
\newcommand{\R}{\mathbb{R}}
\newcommand{\N}{\mathbb{N}}
\newcommand{\Ucal}[0]{\ensuremath{\mathcal{U}}}
\newcommand{\Tcal}[0]{\ensuremath{\mathcal{T}}}

\newcommand{\pr}{\mathbf{P}}
\newcommand{\ex}{\mathbf{E}}
\newcommand{\aut}{\mathrm{aut}}
\newcommand{\frag}{\mathrm{Frag}}

\newcommand{\qr}{\mathrm{qr}}
\newcommand{\exc}{\mathrm{ex}}
\newtheorem{theorem}{Theorem}[section]

\newtheorem{corollary}[theorem]{Corollary}

\newtheorem{lemma}[theorem]{Lemma}
\theoremstyle{definition}

\newtheorem*{nntheorem}{Theorem}
\newcommand{\Ln}{\lim\limits_{n\to \infty}}

\usepackage{float}
\usepackage{hyperref}
\usepackage{enumerate}
\usepackage{enumitem}
\usepackage{chngcntr}
\usepackage{cleveref}
\usepackage{pdfpages}
\usepackage{caption,subcaption,float}
\counterwithout{equation}{section}

\title{Limiting probabilities of first order properties of \\ random sparse graphs and hypergraphs}

\author{
	Alberto Larrauri
	\thanks{
		Universitat Polit\`ecnica de Catalunya, Department of Computer Science. 
		E-mail: {\tt llzalberto@hotmail.com}. 
		Supported by the European Research Council (ERC) under the European Union's Horizon 2020 research and innovation programme (grant agreement ERC-2014-CoG 648276 AUTAR).
	}
	\and
	Tobias Müller
\thanks{
	University of Groningen, Department of Mathematics. 
	E-mail: {\tt tobias.muller@rug.nl}. 
}
\and	Marc Noy
\thanks{
	Universitat Polit\`ecnica de Catalunya, Department of Mathematics. 
	E-mail: {\tt marc.noy@upc.edu}. 
	Supported by the Ministerio de Econom\'{i}a y Competitividad grant MTM2017-82166-P.
}
}

\begin{document}

\maketitle
\begin{abstract}
	Let $G_n$ be the binomial random graph $G(n,p=c/n)$ in the sparse regime, which as is well-known  undergoes a phase transition at $c=1$. 
	Lynch (Random Structures Algorithms, 1992) showed that for every first order sentence $\phi$, the limiting probability that $G_n$ satisfies $\phi$ as $n\to\infty$ exists, and moreover it is an analytic function of $c$. In this paper we consider the closure $\overline{L_c}$ in $[0,1]$ of the set $L_c$ of all limiting probabilities of first order sentences in $G_n$. We show that there exists a critical value $c_0 \approx0.93$ such that $\overline{L_c}= [0,1]$ when $c \ge c_0$, whereas $\overline{L_c}$  misses at least one subinterval when $c<c_0$. We extend these results to random $d$-uniform sparse hypergraphs, where the probability of a hyperedge is given by $p=c/n^{d-1}$. 
	
\end{abstract}

\section{Introduction}

We consider properties of random graphs expressible in the first order (FO) language of graphs, which is first order logic together with an adjacency relation  $E(x,y)$ assumed to be symmetric and antireflexive. Our model is the binomial random graph $G(n,p)$ with vertex set $\{1,\dots,n\}$ and in which every edge is present independently with probability $p$. We focus on the so-called sparse regime $p=c/n$ with $c>0$. It is well-known that 
$G(n,c/n)$ undergoes a phase transition at $c=1$, corresponding to the emergence of the giant component \cite{ER60}. 
The model studied in \cite{ER60} was the uniform model $G(n,M)$ on graphs with $n$ vertices and $M$ edges. However,  the  results we need in this work  can be translated  into the $G(n,p)$ model with the relation  $M=p\binom{n}{2}$. 

Our starting point is the following result by Lynch \cite{lynch}. The notation $G \models \phi$ means that the graph $G$ satisfies the sentence $\phi$. We recall that a sentence is a FO formula without free variables, thus expressing a graph property closed under isomorphism. 

\begin{nntheorem}
For each FO sentence $\phi$, the following limit exists:
$$
p_c(\phi) = 	\lim_{n \to \infty} \pr\{  G_n \models \phi \}.
$$
Moreover, $p_c(\phi)$ is a combination of sums, products, exponentials and a set of constants $\Lambda_c$, hence it is 
an analytic function of $c$. 
\end{nntheorem}

The previous result shows in a strong form that FO logic does not capture the phase transition (see also \cite{ShelahSpencer} for a discussion including monadic second order logic).  

Instead of considering limiting probabilities of single sentences, in this paper we consider the set of all limiting probabilities  $$L_c = \{p_c(\phi)  \colon \phi \hbox{ FO sentence} \},
$$
and its topological closure $\overline{L_c}$ in $[0,1]$.
Our main result is that  there is a transition in the structure of 
$\overline{L_c}$  at a particular value of $c$. 
We say that  $\overline{L_c}$ contains a \emph{gap} if there is at least one subinterval $[a,b] \subseteq [0,1]$ with $a<b$ such that  $\overline{L_c} \cap [a,b] = \emptyset$.

\begin{theorem}\label{th:main}
	Let  $\overline{L_c}$ be the closure of the of limiting probabilities of first order sentences in $G(n,c/n)$.  
Let $c_0 \approx  0.93$ be the unique positive solution of 
\begin{equation}\label{eq:c0}
e^{\frac{c}{2}+ \frac{c^2}{4}} \sqrt{1-c}= \frac{1}{2}.
\end{equation}
Then for every $c>0$ the set $\overline{L_c}$  is a finite union of closed intervals. Moreover, the following  holds:
\begin{enumerate}
	\item $\overline{L_c} = [0,1]$ for $c \ge c_0$.
	\item $\overline{L_c}$ has at least one gap for $0<c < c_0 $.
\end{enumerate}
\end{theorem}

\paragraph{Remark.}
This line of research  was considered in \cite{hmnt2018} for minor-closed classes of graphs under the uniform distribution. For instance, it was shown there that for the class of acyclic graphs (forests), the set $\overline{L_c}$ is the union of 4 disjoint intervals. It was also shown that for every minor-closed class of graphs whose forbidden minors are all 2-connected,  $\overline{L_c}$ is always a finite union of at least two intervals. 

We extend the previous result to random sparse hypergraphs. 
We consider the model $G^d(n,p)$ of random $d$-uniform hypergraphs, where every $d$-edge has probability $p$ of being in $G^d(n,p)$  independently. 
When $p=c/n^{d-1}$ the  expected number of edges $p\binom{n}{d}$ is linear in $n$, justifying the qualifier `sparse'. 
A phase transition also occurs in $G^d(n,c/n^{d-1})$ when $c=(d-2)!$, as shown in \cite{hypergraph-transition}.


\begin{theorem}\label{th:hypergraph}
Let $d\ge3$ be fixed and let  $\overline{L_c}$ be the closure of the of limiting probabilities of first order sentences in $G^d(n,c/n^{d-1})$.
Let $c_0$ be the unique positive solution of 
\begin{equation}\label{eq:r0}
 \exp\left({\frac{c}{2(d-2)!}}\right) \sqrt{1-\frac{c}{2(d-2)!}}= \frac{1}{2}.
\end{equation}
Then for every $c>0$ the set $\overline{L_c}$  is a finite union of intervals. Moreover, the following  holds:
	\begin{enumerate}
		\item $\overline{L_c} = [0,1]$ for $c \ge c_0$.
		\item $\overline{L_c}$ has at least one gap for $0<c < c_0 $.
	\end{enumerate}
\end{theorem}
\noindent
We remark that $c_0 = r (d-2)!$, where $r \approx  0.898$ is the positive solution of $\exp(r/2) \sqrt(1-r)=1/2$. 
As we will see the difference between Equations \eqref{eq:c0} and \eqref{eq:r0} comes from the fact that in graphs we consider cycles of length at least 3, whereas in hypergraphs we have to consider cycles of length at least 2. 

Here is a summary of the paper. In Section \ref{sec:prelim} we review several preliminaries we need on probability, random graphs and logic. In Section \ref{sec:main} we prove Theorem \ref{th:main}, and in Section \ref{sec:hypergrahs} we prove Theorem \ref{th:hypergraph}. To avoid repetition the preliminary results needed for hypergraphs in Section \ref{sec:hypergrahs} are only sketched. 


\section{Preliminaries}\label{sec:prelim}
We start with Brun's sieve for obtaining limiting Poisson distributions \cite[see][Theorem 1.23] {bollobas2001random}.

\begin{lemma}
	\label{th:factorialmethod}
	Let $X_{n,1},\dots, X_{n,k}$ be non-negative integer 
	valued random variables defined over the same probability space.
	Let $\lambda_1,\dots, \lambda_k\in \R$ be non-negative. 
	Suppose that given for $a_1,\dots, a_k\ge0$ it holds that
	\[
	\Ln
	\mathrm{E}\left[
	\prod_{i=1}^{k} \binom{X_{n,i}}{a_i}
	\right] = \prod_{i=1}^k \frac{\lambda_i^{a_i}}{a_i!}.
	\]
	Then the $X_{n,i}$ converge in distribution to 
	independent Poisson variables whose respective means are the $\lambda_i$.
	That is, for any $b_1,\dots, b_k\ge 0$
	\[
	\Ln \pr \left( 
	\bigwedge_{i=1}^k X_{n,i}=b_i
	\right)= \prod_{i=1}^k e^{-\lambda_i} \frac{\lambda_i^{b_i}}{b_i!}.
	\]
\end{lemma}

Next we present several results on the number of cycles in random sparse graphs. 
Let  the number $X_{n,k}$ of $k$-cycles in $G(n,c/n)$. It is easy to show that 
	\[ 	\mathrm{E}[X_{n,k}]\leq \frac{c^k}{2k},
	\]
	and  
	\[
	\Ln \mathrm{E}[X_{n,k}]= \frac{c^k}{2k}.
	\]	
The first part of the next lemma appears already in \cite{ER60} for the $G(n,M)$ model.
The second part is easily proved using the method of moments  \cite[Theorem 1.23]{bollobas2001random}.

\begin{lemma} \label{lem:idependentcycles}

For fixed $k\ge 3$, the number of $k$-cycles $X_{n,k}$ in $G(n,c/n)$ 
is distributed asymptotically as $n\to\infty$  as a Poisson law with parameter $\lambda_k = \frac{c^k}{2k}$. 
Moreover, for fixed $k$ 
the random variables $X_{n,3}, \dots, X_{n,k}$ are asymptotically independent. 
\end{lemma}

We set 
\begin{equation}\label{eq:fc}
f(c)=\frac{1}{2}\ln\frac{1}{1-c} -\frac{c}{2} - \frac{c^2}{4}.
\end{equation}
This is a function defined on $(0,1)$ that plays an important role in our results. The function is $e^{-f(c)}$ the limiting probability that $G(n,c/n)$ is acyclic; see Figure \ref{fig:plot} for a plot. 

\begin{figure}
\begin{center}
\includegraphics[scale=0.4]{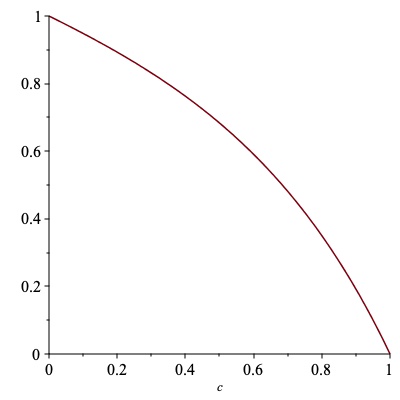}
\caption{The probability that $G(n,c/n)$ has no cycles as a function of $c$.}\label{fig:plot}
\end{center}
\end{figure}

\begin{corollary}\label{cor:acyclic}
	When  $c<1$  the expected number of cycles in $G(n,c/n)$ is $f(c)$.
	
	Moreover,  the limiting probability as $n\to\infty$ that $G(n,c/n)$ contains no cycle  is 
	$$e^{-f(c)}=e^{\frac{c}{2}+\frac{c^2}{4}}\sqrt{1-c}.
		$$
\end{corollary}

\begin{proof}
We have
$$
\lim_{n\to\infty} \ex\left[\sum_{k\ge3}X_{n,k}\right]= 
\sum_{k\ge3} \frac{c^k}{2k} = f(c).
$$
The second statement follows from Lemma \ref{lem:idependentcycles}.
\end{proof}

A graph is unicyclic, or a unicycle,  if it is connected and has a unique cycle. A unicyclic graph has a simple structure: it consists of a cycle of length at least 3 and a collection of rooted trees attached to its vertices.  The \emph{size} of a unicyclic graph is the number of edges, which is equal to the number of vertices (we use this convention because for hypergraphs it is more convenient to define size as the number of edges). 
We denote by $\Ucal$ the family of unlabeled graphs whose connected components are all unicyclic, and we let $\Ucal_n= \{ H\in \Ucal \, : \, |H|=n\}$ and $\Ucal_{\leq n}=\bigcup_{i=1}^n \Ucal_i$.

The following is  well-known; see  \cite[Lemma 2.10]{frieze2016introduction} for a proof in the $G(n,p)$ model. 
We say that a property holds asymptotically almost surely (a.a.s.) if it holds with probability tending to 1 as $n\to\infty$. 
\begin{lemma}\label{th:nogiantcomponent}
	Let $p(n)\sim c/n$ with $0 < c <1$. Then a.a.s all the  connected component of
	$G(n,p)$ are either trees or  unicycles. 		
\end{lemma}

Given a graph $G$, we define its \emph{fragment} $\frag(G)$ as the union of all the  unicyclic components in $G$. We
will write $\frag_n$ to denote the fragment of $G(n,p)$. The following result states that below the critical value $c=1$ the expected size of $\frag_n$ is asymptotically bounded. 

\begin{lemma} \label{thm:sizefragment}
	Let $p(n)\sim c/n$ with $0 < c <1$. 
	Then $\Ln \mathrm{E}[|\frag_n|]$ exists and
	is a finite quantity.	
\end{lemma}
\begin{proof}
This is done in \cite[Theorem 5d]{ER60} for the uniform model and in greater detail in  	\cite[Lemma 2.11]{frieze2016introduction} for the binomial model. For future reference we sketch the main ingredients in the proof. 
	
Let	$Y_{n,i}$ be the random variable equal to 
the number of unicyclic components 
in $G_n$ that contain exactly $i$ edges.
Then  one proves that for  $k$ large enough and  $n\ge0$
	\[
	\mathrm{E}[Y_{n,k}]\leq (c e^{1-c})^k e^{c/2}.
	\]
Furthermore, for all  $k\geq 3$
	\[
	\Ln \mathrm[Y_{n,k}]= C(k,k) (c e^{-c})^k,
	\]
where $C(k,k)$ denotes the number of labeled unicyclic graphs on $k$ vertices. Then the statement follows from  the Dominated Convergence Theorem.
\end{proof}

Let $\phi$ be a first order sentence. We recall that the quantifier rank $\qr(\phi)$ is the maximum number of nested quantifiers in $\phi$. 
It is shown in \cite{lynch} that whether $G_n$ satisfies $\phi$ or not  depends only a.a.s. on the  induced unicycles in $G_n$ of diameter at most $3^k$, where $k=\qr(\phi)$ (see Theorems 4.7, 4.8 and 4.9 in \cite{lynch}). 
This, together with the fact that for $c<1$ a.a.s. the connected components of $G_n$ are  either trees or unicycles  (see Theorem \ref{th:nogiantcomponent}), implies the following:

\begin{lemma} \label{lem:zeroonefragment}
	Let $p(n)\sim c/n$ with $0 < c <1$. 
	Let $\phi$ be a FO sentence and let $H\in \Ucal$.
	Then 
	\[
	\Ln \pr \big( \,
	G(n,p) \models \phi \, \big| \, \frag_n\simeq H\,
	\big)=\text{$0$ or $1$}.
	\]
Moreover, the value of the limit depends only on 
$\phi$ and $H$, and not on $c$. 
\end{lemma}

Because of \Cref{th:nogiantcomponent}, when $0<c< 1$ a.a.s. all cycles in $G_n$
are contained in  unicyclic components. Since the  expected number $f(c)$ of cycles in $G_n$
is asymptotically bounded we obtain the following.

\begin{corollary}\label{cor:nocyclesincomplexcomponents}
	Let $p(n)\sim c/n$ with $0<c<1$, and   let
	$Z_n$ be the random variable equal to  the number of  cycles in $G(n,p)$ that
	belong to connected components that are not trees or unicycles. 
	Then 
	\[
	\Ln \mathrm{E}[Z_n]=0.
	\]
\end{corollary}

Let $\aut(H)$ denote the number of automorphisms of a graph $H$.

\begin{lemma} \label{lem:unicycliccomponents}
	Let $p(n)\sim c/n$ with $c>0$.
	Let $T$ be a finite set of unlabeled
	unicycles. For each $H\in T$ let 
	$X_{n,H}$ be the random variable equal to  the number of connected components in $G(n,p)$ isomorphic to $H$, and let
	$\lambda_H=\frac{(e^{-c}c)^{|H|}}{\aut(H)}$. 	Then
	\[
	\Ln \pr\left(
	\bigwedge_{H\in T} X_{n,H}=a_H
	\right) = \prod_{H\in T} 
	e^{-\lambda_H}\frac{\lambda_H^{a_H}}{a_H!}.
	\]
	In other words, the $X_{n,H}$ converge in distribution
	to independent Poisson variables with respective 	means~$\lambda_H$. 
\end{lemma}
\begin{proof}
The proof is  a slight modification of Theorem 4.8 in \citealp{bollobas2001random}.
	It follows from a straightforward  application of Theorem 	\ref{th:factorialmethod}.	
\end{proof}

We also need a classical result conjectured by Kakeya \cite{kakeya} and later 
proven in \cite{nymann2000paper} on 
the set of subsums of a convergent series of non-negative terms. 
\begin{lemma}
	\label{kakeya}
	Let $\sum _{n\ge 0}p_n$ be a convergent series of non-negative real numbers. Then 
	the following are equivalent:
	\begin{itemize}
		\item[(1)] $p_i \le \sum_{j>i} p_j$ for all $i \ge 0$.
		\item[(2)]
		\[
		\left\{\sum_{i \in A} p_i  \colon  A \subset  \mathbb{N}\right\}
		= \left[0, \sum_{n=0}^{\infty} p_n\right].
		\]
	\end{itemize}
	Moreover, if the condition $p_i \le \sum_{j>i} p_j$ holds for all values of $i$ 
	large enough, 
	then the set $\left\{\sum_{i \in A} p_i  \colon  A \subset  \mathbb{N}\right \}$
	is a finite union of intervals. 
\end{lemma}

\section{Proof of Theorem \ref{th:main}}\label{sec:main}


Before proceeding with the proof we outline the  main components in the proof:
\begin{itemize}
	\item For $c\geq 1$ we can approximate any  $p\in [0,1]$ with
	probabilities of statements of the form ``there are at most $l$ cycles
	of length at most $k$ in $G(n,c/n)$'' (\Cref{sec:nogap_supercritical}).
	
	\item We prove that the constant $c_0$ given by \Cref{eq:c0} satisfies 
	that $\overline{L_c}$ contains at least one gap whenever $c<c_0$ (\Cref{sec:c0}).
	
	\item For $c<1$ we show that 
	\[
	\overline{L_c} = \left\{ \sum_{H \in T} p_H(c) \colon T\subseteq \Ucal	
	\right\},	
	\]
	where $p_H(c)= \Ln \pr\big(\frag_n \simeq H \big)$ is given by \Cref{eq:prob-fragment}.
	In other words, $\overline{L_c}$ coincides with the set of  subsums of the convergent  series
	$\ds\sum_{H\in \Ucal}p_H(c)=1$.
 
	\item Using Kakeya's Criterion (\Cref{kakeya}) we show that the set $\overline{L_c}$  of 
	subsums of 	$\sum_{H\in \Ucal}p_H(c)$
 is always  a finite union of intervals (\Cref{sec:finitegaps}). 
		
	\item Using Kakeya Criterion  (\Cref{kakeya}) once again we prove that $\overline{L_c}$  is the whole interval
	$[0,1]$ for $c_0\leq c<1$ (\Cref{sec:nogap}). 
\end{itemize}

\subsection{No gap when $c \geq 1$} \label{sec:nogap_supercritical}

Let $X_k$ be as before  the number of cycles of length  $k$ in $G(n,c/n )$, which is 
asymptotically $\hbox{Po}(c^k/(2k)$). Moreover, 
for fixed $k$, the random variables $X_3, \dots, X_k$
are asymptotically independent by \Cref{lem:idependentcycles}. Hence for fixed $k$,  
$$ X_{\leq k} = X_3 + \dots + X_k \xrightarrow[n\to\infty]{\text{d}} \hbox{Po}\left( \sum_{i=3}^k \frac{c^k}{2k} \right).
$$
Since $c \geq 1$ the mean $\sum_{i=3}^{k}c^k/2k$ is not bounded as $k$ grows to infinity
so we can pick $k$ such that this mean is as large as we like.
Note that for any $k$ and $a$ the property that $X_{\leq k} \leq a$ 
can be expressed in FO logic.
By the central limit theorem we have, for any fixed $x \in \mathbb{R}$
$$ \pr( \hbox{Po}(\mu) \leq \mu + x \sqrt{\mu} ) \xrightarrow[\mu\to\infty]{} \Phi(x), $$
where $\Phi(x)= \frac{1}{\sqrt{2\pi}}\int_{-\infty}^{x} e^{-t^2/2}dt$ is the c.d.f. of the standard normal law. 

For $0 < p < 1$ and $\epsilon > 0$ we can find $x$ such that $\Phi(x) = p$, a value 
$\mu_0$ such that $\pr(\mathrm{Po}(\mu) \leq \mu + x \sqrt{\mu} ) \in (p-\epsilon,p+\epsilon)$ for all $\mu\geq\mu_0$, and then 
finally a $k$ such that $ \sum_{i=3}^k \frac{c^k}{2k} \geq \mu_0$.
Hence there exists a FO property $\phi$ with limiting probability within $\epsilon$ of $p$.

\subsection{At least one gap when $c<c_0$} \label{sec:c0}

Let $e^{-f(c)}$ be as in Corollary  \ref{cor:nocyclesincomplexcomponents}  the limiting probability that $G_n=G(n,c/n)$ is acyclic. By elementary calculus  $e^{-f(c)}$
is strictly decreasing for $c\in [0,1]$ (see Figure \ref{fig:plot}), and by definition of $c_0$ we have $e^{-f(c_0)}=1/2$. 

Fix $c<c_0$, and let $p(n)\sim c/n$. 
We are going to show that $\overline{L_c}$ 
has a gap around $1/2$.
Let $A_n$ be the event that $G_n$ is acyclic and 	let $\phi$ be a FO sentence and let $p(\phi)=  \Ln\pr\big(G_n\models \phi)$. Then 
	\begin{align}\nonumber
	p(\phi) &= \Ln \pr\big(G_n\models \phi \big| A_n \big)  \pr(A_n) + 
	\pr\Big(G_n\models \phi \big| \neg A \Big) \pr( \neg A_n)\\
	\label{eq:aux1}
	&= \Ln \pr\big(G_n\models \phi\, \big|\, A_n\big) e^{-f(c)} +
	 \pr\big(G_n\models \phi\, \big|\, A_n\big) (1-e^{-f(c)}).	 
	\end{align}
	Because of Lemma \ref{lem:zeroonefragment} we have
	\begin{equation*}
	\Ln \pr\big(G_n\models \phi\, 
	\Big| \, A_n \big)=\text{$0$ or $1$}.
	\end{equation*}
If the last limit equals 0 then from 
	 \Cref{eq:aux1} we obtain that $p(\phi)\leq 1- e^{-f(c)}$. 	Otherwise, if the limit is 1, we have $p(\phi)\geq e^{-f(c)}$.
	Since $e^{-f(c)}$ is strictly decreasing, $c<c_0$  and $e^{f(c_0)}=1/2$ it follows that $e^{-f(c)}>1/2$.
	As a consequence $1-e^{-f(c)}<1/2<e^{-f(c)}$ and $(1-e^{-f(c)},e^{-f(c)})$ is a gap
	of $\overline{L_c}$.

\subsection{Asymptotic distribution of the fragment for $c<1$ and its consequences}\label{sec:frag_dist}

We compute below that the asymptotic probability that the fragment $\frag_n$ is isomorphic to a given union $H$ of unicycles.

\begin{lemma} \label{lem:prob-fragment}
	Let $p(n)\sim c/n$ with $0<c<1$, and let $H\in \Ucal$.
	Then
	\begin{equation}\label{eq:prob-fragment}
	\Ln \pr \big(
	\frag_n \simeq H 
	\big) = e^{-f(c)} \frac{(e^{-c}c)^{|H|}}{\aut(H)}.
	\end{equation}
\end{lemma}
\begin{proof}
	
Fix such an $H$.
Let $U_1,U_2,\dots, U_i,\dots$ be an enumeration of all 	unlabeled unicycles ordered by non-decreasing size. For each $i$ let $a_i$ 	be the number of connected components of $H$ that are copies of $U_i$,
and let $W_{n,i}$ be the random variable equal to  the number of connected components in $G_n$ that are isomorphic to $U_i$. Clearly	$\frag_n\simeq H$ if and only if $W_{n,i}=a_i$ for all $i$. 
	Thus,
	\[
	\Ln \pr\big(\frag_n\simeq H\big)= 
	\Ln \pr\big( \bigwedge_{i=1}^\infty W_{n,i}=a_i\big).
	\]
	First, we are going to show that 
	\begin{equation}\label{eq:aux10}
	\Ln \pr \big( \bigwedge_{i=1}^\infty W_{n,i}=a_i\big)=
	\lim\limits_{j\to \infty}
	\Ln \pr \big( \bigwedge_{i=1}^j W_{n,i}=a_i\big).
	\end{equation}
	Fix $\epsilon>0$. 
	For each $k$ let $X_{n,k}$ be the random variable
	that counts the unicyclic connected components of $G_n$
	with exactly $k$ edges.	By \Cref{thm:sizefragment}
	we have that for some $k_0$
	\[ 
	\Ln \sum_{l=k_0}^{\infty }\mathrm{E}\big[ X_{n,l} \big] \leq 
	\epsilon.
	\]
	Let $k_1$ be the maximum number of edges in a connected component of 
	$H$, and let $k = \max(k_0, k_1+1)$. Finally, fix $j_0$ such that 
	$e(U_j)>k_1$ for any $j\geq j_0$. 
	Let $j\geq j_0$. Then $\frag_n \simeq H$ if and only if 
	\[
	\bigwedge_{i=1}^{j} W_{n,i}=a_i, \qquad \text{ and }
	\sum_{\ell=k}^\infty X_{n,l}=0.
	\]
	Using Markov's inequality we get 
	\[
	\Ln \pr\big( \sum_{\ell=k}^\infty X_{n,\ell}\geq 1  \big) \leq \epsilon.
	\]
	And in consequence,
	\[
	\left|\Ln \pr\big( \bigwedge_{i=1}^\infty W_{n,i}=a_i\big)
	- \lim\limits_{j\to \infty}
	\Ln \pr\big( \bigwedge_{i=1}^j W_{n,i}=a_i\big)
	\right|
	\leq \epsilon,
	\]
	proving \cref{eq:aux10}.
	
	Using \Cref{lem:unicycliccomponents} we get
	\[
	\lim\limits_{j\to \infty}
	\Ln \pr\big( \bigwedge_{i=1}^j W_{n,i}=a_i\big)=
	\prod_{i=1}^\infty e^{-\lambda_i} \frac{\lambda_i^{a_i}}{a_i!},
	\]
	where $\lambda_i=\frac{(c e^{-c})^{|U_i|}}{\aut(U_i)}= 
	\Ln \mathrm{E}[W_{n,i}]$. Using \Cref{cor:nocyclesincomplexcomponents} 
	together with the dominated convergence theorem  as in the proof of \Cref{thm:sizefragment} we obtain
	\[
	\sum_{i=1}^{\infty} \lambda_i = f(c),
	\]
	and as a consequence 
	\[
	\prod_{i=1}^\infty e^{-\lambda_i}= e^{-f(c)}=e^{-f(c)}.
	\]
	
	Since $
	\sum_{i=1}^\infty |U_i| a_i = |H|$ and $	\prod_{i=1}^{\infty} \aut(U_i)^{a_i} a_i!= \aut(H)$,
 we finally get
	\[
	\prod_{i=1}^\infty \frac{\lambda_i^{a_i}}{a_i!}
	=\frac{(c e^{-c})^{|H|}}{\aut(H)},
	\]
	and the result follows. 
\end{proof}

Given $H\in \Ucal$ we define $p_H =p_H(c) = \Ln \pr \Big( \frag_n\simeq H  \Big)$.
The following is a direct consequence of the fact that the expected size of $\frag_n$ is
bounded.
\begin{lemma}\label{lem:limitexchangefragment}
	Let $p(n)\sim c/n$ with $0<c<1$, and let $\Tcal\subset \Ucal$. Then
	\[
	\Ln \pr\left(
	\bigvee_{H\in \Tcal} \frag_n\simeq H
	\right) = \sum_{H \in \Tcal} p_H.
	\]
	In particular, $\sum_{H \in \Tcal} p_H=1$. 
\end{lemma}
\begin{proof}
	If $\Tcal$ is finite then the statement is clearly true, 	since the events $\frag_n\simeq H$ are disjoint for different $H$.
	Suppose otherwise. Let $H_1,\dots, H_i,\dots$ be an enumeration of $\Tcal$ by non-decreasing size. 
	Fix $\epsilon>0$.
	Let $m=\Ln \mathrm{E}\big[ |\frag_n| \big]$, and let $M=m/\epsilon$. 
	Then there exists  $j_0$ such that  $E(H_j)\geq M$  for all $j\geq j_0$.
	Using Markov's inequality we obtain that for any $j\geq j_0$
	\begin{align*}
	&\Ln \left|
	\pr\left( 
	\bigvee_{H\in T} \frag_n \simeq H
	\right) - \sum_{i=1}^{j} p_{H_i}	
	\right|	\\
	&=\Ln \left|
	\pr\left( 
	\bigvee_{H\in T} \frag_n \simeq H
	\right) - \pr\left(
	\sum_{i=1}^{j} \frag_n \simeq H \right)
	\right| \\ &\leq \Ln \pr\Big(
	|\frag_n|>M
	\Big) \leq \epsilon.
	\end{align*}
	As our choice of $\epsilon$ was 
	arbitrary this proves the statement. 
\end{proof}

We define $S_c$ as the set of subsums of $\sum_{H\in \Ucal}p_H(c)$,
\[
S_c := \Big\{ \sum_{H \in \Tcal} p_H(c) \colon  \Tcal\subseteq \Ucal
\Big\}.
\]

\begin{theorem} \label{th:partialsums}
	Let $0<c<1$. Then
	$\overline{L_c}=S_c$. 
\end{theorem}
\begin{proof}
	We first show that  $\overline{L_c}\subseteq S_c$.
	It is a known fact \cite{kakeya}, 
	\cite{hornich1941beliebige}, \cite{nymann2000paper}
	that $S_c$ is a perfect set,  in particular
	$S_c$ is closed and $\overline{S_c}=S_c$. Thus it is enough  to show that
	$L_c\subset S_c$.
	
	Fix a FO property $P$. We want to prove that
	$\Ln \pr(P(G_n))$ lies in $S_c$. That is, we want 
	to show that for some $\Tcal\subseteq \Ucal$
	\begin{equation} \label{eq:aux2}
	\Ln \pr\big(P(G_n) \big)= \sum_{H\in \Tcal} p_H.
	\end{equation}
	Let $H\in \Ucal$.
	First we will prove that 
	\begin{equation}\label{eqn:aux}
	\Ln \pr \big( P(G_n) \big)
	= \sum_{H\in \Ucal} \Ln \pr \big( \frag_n \simeq H \big)
	\pr \big( P(G_n)) \, | \, \frag_n\simeq H \big).
	\end{equation}
	Let $H_1,\dots, H_i, \dots$ be an enumeration of $\Ucal$.
	Fix an arbitrarily small real constant $\epsilon>0$.
	Notice that the events of the form $F_n\simeq H_i$ are disjoint for each
	$i$. So we obtain:
	\[
	\Ln \pr\big( P(G_n)\big)=
	\Ln \sum_{i=1}^\infty \big( \frag_n \simeq H_i \big)
	\pr\big( P(G_n) \, | \, \frag_n\simeq H_i \big).	
	\] 
	Let $m=\Ln \mathrm{E}\big[|\frag_n|\big]$ and let $M=m/\epsilon$.
	There exists some $j_0\in \N$ such that $|H_i|\geq M$ 
	for all $i\geq j_0$. As a consequence, for all $j\geq j_0$,
	\[
	\Ln \sum_{i=j}^{\infty} \pr\big( \frag_n \simeq H_i \big) \leq 
	\epsilon.	
	\]
	And we obtain
	\begin{align*}
	\big|\Ln \pr\big( P(G_n) \big) -
	\sum_{i=1}^j \Ln \pr\big( \frag_n \simeq H_i \big)
	\pr\big( P(G_n) \, | \, \frag_n\simeq H_i \big)\big|&\\=
	\Ln \sum_{i=j}^{\infty} \pr\big( \frag_n \simeq H_i \big) \leq 
	\epsilon &.
	\end{align*}
	This proves \Cref{eqn:aux}. Because of \Cref{lem:zeroonefragment}
	for all $H\in \Ucal$
	\[
	\Ln \pr \big( P(G_n) \, | \, \frag_n\simeq H \big)
	= 0 \text{ or } 1.
	\]
	Let $
	T=\{i\in \N \, | \, \Ln \pr
	\big( P(G_n) \, | \, \frag_n\simeq H_i \big) = 1 \,	
	\}$.
	Using \Cref{eqn:aux} we obtain  
	\begin{align}
	\sum_{i=1}^\infty \Ln \pr\big( \frag_n \simeq H_i \big)
	\pr\big( P(G_n) \, | \, \frag_n\simeq H_i \big)
	= \sum_{i\in T} p_{H_i}.
	\end{align}
	This proves \Cref{eq:aux2} and as a consequence 
	$\overline{L_c}\subset S_c$.\par
	
 Now we proceed to prove $S_c \subset \overline{L_c}$.		
	Let $T \subset \Ucal$, and let $\epsilon>0$ be
	an arbitrarily small real number. We will show that 
	there exists a FO property $P$ such that
	\begin{equation}\label{eq:aux3}
	\left| \Ln \pr\big( P(G_n)\big) 
	- \sum_{H\in T} p_H \right| \leq \epsilon.
	\end{equation}
	First, notice that 
	because of \Cref{lem:limitexchangefragment}, 
	\[
	\sum_{H\in T} p_H = \Ln \pr\left( 
	\bigvee_{H\in T} p_H
	\right).
	\]
	We define the property $Q$ in the following way
	\[
	Q(G):= \bigvee_{H\in T} \frag(G)\sim H.
	\]	
	Let $m:=\Ln \mathrm{E}\big[|\frag_n|\big]$, and let $M=2m/\epsilon$.
	Then using Markov's inequality:
	\begin{equation} \label{eq:aux4}
	\Ln \pr \big( \frag_n \in \Ucal_{\leq M}  \big)\geq 1 - \epsilon/2.
	\end{equation}
	Also, using that $\frag_n\neq \frag_n^M$ implies that $|\frag_n|\geq M$,
	\begin{equation}\label{eq:aux5}
	\Ln \pr \big(\frag_n \neq \frag^M_n\big) \leq \pr \big( \frag_n \in \Ucal_{\leq M}  \big)\leq 
	\epsilon/2.
	\end{equation}
	Let $T^\prime=T\cap \Ucal_{\leq M}$. As $\Ucal_{\leq M}$ is a finite 
	set so is $T^\prime$. Define the properties
	$Q^\prime$ and $P$ as
	\[
	Q^\prime(G)=\bigvee_{H\in T^\prime} \frag(G)\sim H, \quad \text{ and }
	P(G)= \bigvee_{H\in T^\prime} \frag^M(G)\sim H.
	\]
	Notice that $P$ can be expressed in FO logic. 
	Also, the implications $Q^\prime\implies Q$ and $Q^\prime \implies P$ hold. 
	If $Q(G)$ holds and $\neg Q^\prime(G)$ holds as well then in particular
	$|\frag(G)|\geq M$. The same happens if $P(G)$ and $\neg Q^\prime(G)$ hold
	at the same time. Joining everything we get
	\begin{align*}
	&\left| \Ln \pr\big( P(G_n)\big) 
	- \sum_{H\in T} p_H \right|= 
	\left| \Ln \pr\big( P(G_n)\big) 
	- \Ln \pr\big( Q(G_n)\big)  \right|\\&
	\leq 
	\left| \Ln \pr\big( P(G_n)\big) 
	- \Ln \pr\big( Q^\prime(G_n)\big)  \right|+
	\left| \Ln \pr\big( Q^\prime(G_n)\big) 
	- \Ln \pr\big( Q(G_n)\big)  \right|\\&
	\leq 2\pr\big(
	|\frag_n|\geq M
	\big) \leq \epsilon.	
	\end{align*}
\end{proof}

\subsection{$\overline{L_c}$  is always a finite union of intervals}\label{sec:finitegaps}

Since we have shown that  there is no gap for $c\ge 1$ we only need to consider $c<1$. 
Let $H_1,\dots, H_n,\dots$ be an enumeration of $\Ucal$ such that
$p_{H_i}(c)\leq p_{H_j}(c)$ for all $i\leq j$. We shorten $p_{H_i}$ to $p_i$. Because of \Cref{kakeya} proving that $\overline{L_c}$ is a finite union of 
intervals amounts to showing that for all $i$ large enough
\begin{equation}\label{eq:finitegaps}
p_i \leq \sum_{j>i}^{\infty} p_j.
\end{equation}

Let $f=f(c)$ be as defined in \Cref{eq:fc}, and let $s=c e^{-c}$, and  notice
that as $c<1$ we have $s<1$ as well. We can rewrite the $p_i$ given by \Cref{eq:prob-fragment} as
\[
p_i= e^{-f}\frac{s^{|H_i|}}{aut(H_i)}.
\]
For $i\ge1$ let $k(i)$ be the least integer such that
\[ e^{-f} s^{k(i)-1} \geq p_i > e^{-f} s^{k(i)}. 
\]
Notice if $k\geq k(i)$ and  $H_j\in \Ucal_k$ then 
$p_j< e^{-f}s^k < p_i$ because $\aut(H_j)\geq 1$. For the same reason we also obtain
that $|H_i|\leq k(i)-1$. Hence to prove \eqref{eq:finitegaps}  it is sufficient to show that
\begin{equation}
p_i \leq \sum_{k\geq k(i)} \sum_{H_j \in \Ucal_k} p_j.
\end{equation}

Let $C_{x,y}$ denote the graph in $\Ucal$ consisting  of a cycle of length $x$ with a path of length $y$ attached to one 
of its vertices.If $y=0$ then $\aut(C_{x,y})=2x$, and $\aut(C_{x,y})=2$ otherwise. 
Let $T_{x,y,z}$ be the graph consisting of a triangle with paths of length $x$, $y$,  and $z$ attached to its three 
vertices. Note that $\text{aut}(T_{x,y,z}) = 1$ if $x,y,z$ are distinct, $\text{aut}(T_{x,y,z}) = 6$ if $x=y=z$, and
$\text{aut}(T_{x,y,z}) = 2$ otherwise. 
It is easy to see that  
$C_{3,k-3}, C_{4,k-4},\dots, C_{k-1,1}$ together with
$T_{0,1,k-4}, T_{0,2,k-5}, \dots, T_{0,\lfloor(k-3)/2\rfloor, \lceil(k-3)/2\rceil}$
form a family of different elements  of $\Ucal_k$.
We have  that for $k\ge3$
\[
\sum_{i=3}^{k-1} p_{C_{i,k-i}}= e^{-f} s^k  \frac{k-3}{2}.
\]
If $k$ is odd $T_{0,\lfloor(k-3)/2\rfloor, \lceil(k-3)/2\rceil}$ has two automorphisms, and
the remaining $T_{i,k-3-i}$ with $i\geq 1$ each have only one automorphism.
If $k$ is even then all of $T_{0,1,k-4}, T_{0,2,k-5}, \dots, T_{0,\lfloor(k-3)/2\rfloor, \lceil(k-3)/2\rceil}$
have exactly one automorphism.
This gives
\[
\sum_{i=1}^{\lfloor(k-3)/2 \rfloor} p_{T_{0,i,k-3-i}} =e^{-f} s^k   
\frac{k-4}{2}, \quad \text{ for  $k\geq 4$ }.
\]
Using the last two equations it follows  that for  $k\geq 4$
\begin{equation}\label{eq:UkBound}
\sum_{H \in \Ucal_k} \frac{1}{\text{aut}(H)} \geq e^{-f} s^k   \frac{2k-7}{2}. 
\end{equation}
Hence if $i$ is such that $(2k(i)-7)/2 > 1/s$ (that is, $k(i)>1/s+7/2$) 
then
\[ 
\sum_{j>i} p_j \geq \sum_{H_j \in \Ucal_{k(i)}} p_j
\geq e^{-f} s^{k(i)}   \frac{2k-7}{2} > e^{-f} s^{k(i)-1} \geq p_i. \]
\noindent
Note that $k(i)>1/s+ 7/2$ whenever $|H_i|+1 \geq 1/s +7/2$, and this is  true for  sufficiently large $i$.
We have seen that, for any $0<c<1$, it is indeed the case that $p_i < \sum_{j>i} p_{j}$ for all sufficiently large $i$, as was to be proved.
%
%
%
%

\subsection{No gap when $c_0 \le c < 1$} \label{sec:nogap}

Fix $c\in [c_0 , 1)$. Notice that in this case
then $s=ce^{-c}$ satisfies 
\[
\frac{1}{3} < s < \frac{1}{e}. 
\]
Let $i$ be such that $k(i)\geq 4$. Then, using \eqref{eq:UkBound} we obtain
\[
\sum_{j>i} p_j \geq \sum_{k\geq k(i)} \sum_{H_j\in \Ucal_k} p_j \geq
\sum_{k\geq k(i)} e^{-f} s^k  \frac{2k-7}{2}.
\]
And using 
$\sum_{k=0}^{\infty} a^k(b+ck)  = \frac{b}{1-a} + \frac{ca}{(1-a)^2}
$
together with $s> 1/3$ we obtain that
\[
\sum_{j>i} p_j \geq e^{-f} s^{k(i)} \left(\frac{2k(i)-7}{2(1-s)} + \frac{s}{(1-s)^2}\right)
\geq e^{-f} s^{k(i)} \frac{3k(i)-9}{2}. 
\]
In particular, since $\frac{3k-9}{2} \geq 3 > 1/s$ for all $k\geq 5$, 
if $p_i \leq s^{4}$ then $p_i < \sum_{j>i} p_j$. 
As a consequence, if $|H_i| \geq 4$ then $p_i < \sum_{j>i} p_j$.

The only two cases left to consider are the ones when $H_i$ is
either the empty graph or the triangle. 
If $H_i$ is the empty graph 
then necessarily $i=1$ because the empty graph is the most likely fragment.
By the definition of $p_0$ critically we have 
$p_1\leq 1/2$ if $c\ge c_0$, hence $p_1\leq \sum_{j>1} p_j$. If  $H_i$ is the triangle graph, then 
$p_i=e^{-f} s^3/6$ and
\begin{align*}
\sum_{j>i}p_j  = \sum_{k\geq 4} \sum_{H_j\in \Ucal_k} p_j
&\geq  \sum_{k\geq 4} e^{-f} s^k  \frac{2k-7}{2} 
\geq e^{-f} s^4 \frac{3}{2} \geq e^{-f} s^3 \frac{1}{6}= p_i,
\end{align*}
as needed. 

Thus $p_i \leq \sum_{j>i} p_j$ for every $i$, as we needed to prove.
\section{Proof of Theorem \ref{th:hypergraph}}\label{sec:hypergrahs}
  
Recall that we consider the model $G^d_n= G^d(n,p=c/n^{d-1})$ of  random $d$-uniform hypergraphs
where each $d$-edge has probability $p=c/n^{d-1}$ of being in $G^d_n$  independently, with $c>0$. Throughout 
this section we consider $d\ge 3$ as being fixed and we will refer to ``$d$-uniform hypergraphs''
simply as hypergraphs. 
The FO language of $d$-uniform hypergraphs is the FO language with 
a  $d$-ary relation which is anti-reflexive and completely symmetric. Analogously to the case of 
graphs, this relation symbolizes the adjacency relation in the context of $d$-uniform hypergraphs.  

The following is an analog of Lynch's convergence law for random hypergraphs and can be found in 
\cite[Proposition 6.4]{saldanha2016spaces} and in more detail for more general relational structures
in \cite{alberto}.

\begin{nntheorem}
	Let $p(n)\sim c/n^{d-1}$. 
	Then for each FO sentence $\phi$, the following limit exists:
	$$
	p_c(\phi) = 	\lim_{n \to \infty} \pr\big( G^d_n \models \phi \big).
	$$
	Moreover, $p_c(\phi)$ is a combination of sums, products,
	exponentials and a set of constants $\Lambda_c$, hence it is 
	an analytic function of $c$. 
\end{nntheorem}

As before we consider the set
\[
L_c=\left\{ \Ln \pr(G_n^d \models \phi ) \, : \, \text{ $\phi$ FO sentence, and }
p(n)\sim c/n^{d-1} \right\}.
\]

\subsection{Hypergraph preliminaries}

Given a hypergraph $H$,  we denote the  set of vertices by $V(H)$ and the 
set of edges by $E(H)$. As in \cite{karonski2002phase} 
we define the \emph{excess} $\exc(H)$ of $H$ as the quantity  
\[
\exc(H)= (d-1)|H| - |V(H)|.
\]
It is easily seen that the minimum excess of a connected hypergraph is $-1$.
A \emph{tree} $T$ is a connected hypergraph satisfying $\exc(T)=-1$. Equivalently, a tree is a
hypergraph that can be obtained  gluing edges repeatedly to an initial vertex in such a way that each new edge
intersects the hypergraph obtained so far in exactly one vertex. 
A \emph{unicycle}  is a connected hypergraph of excess 0, and  a \emph{cycle} is a minimal unicycle. 
Equivalently a cycle is a connected hypergraph $H$ where every  edge shares exactly two vertices with
the remaining  edges, and a unicycle is a cycle with disjoint trees attached to each vertex. A $k$-cycle is a cycle with $k$ edges. 

It is shown in \cite{hypergraph-transition} that a phase transition in the structure of $G_n^d$ 
 occurs when $c= (d-2)!$, similar to the one for random graphs.
In particular, we have the following results \cite[Theorem 3.6]{hypergraph-transition}.

\begin{lemma} \label{lem:hyp_no_complex_component}
	Let $p(n)\sim c/n^{d-1}$ with $0<c<(d-2)!$. Then a.a.s. all connected components of 
	$G^d_n$ are either trees or unicycles. 	
\end{lemma}

The proofs of the next results are very similar to those for graphs presented in Section \ref{sec:prelim} and are omitted. 

\begin{lemma} \label{lem:hyp_idependent_cycles}
	Let $p\sim c/n^{d-1}$ with $c>0$. For each $k\ge 2$,
	let $X_{n,k}$ be the random variable equal to the  number of $k$-cycles 
 in $G^d_n$, and let 
	 $\lambda_k= \left(\frac{c}{(d-2)!}\right)^k$.
	Then for fixed  $k\geq 2$
	\begin{itemize}
		\item[(1)] $\mathrm{E}\big[ X_{n,k}\big]\leq \lambda_k$,
		\item[(2)] $\Ln \mathrm{E}\big[ X_{n,k}\big]= \lambda_k$, 
		\item[(3)] $X_{n,k}$ converges in distribution to a Poisson variable
		with mean $\lambda_k$ as $n\to\infty$. 
	\end{itemize}
	Furthermore, for any fixed $k\geq 2$ the variables $X_{n,2},\dots, X_{n,k}$
	are asymptotically independent.  
\end{lemma}

\begin{corollary}\label{lem:hyp_acyclic_prob}
	Let $p\sim c/n^{d-1}$ with $c>0$. Set 
	\begin{equation}\label{eq:hyp_fc}
	f(c)=\sum_{k\geq 2} \left(\frac{c}{(d-2)!}\right)^k \frac{1}{2k}=
	\frac{1}{2} \ln\frac{1}{ 1 - \frac{c}{(d-2)!}}  -  \frac{c} {2(d-2)!}.
	\end{equation}
	Let $X_n$ be the random variable equal to  the total number of cycles in $G^d_n$. Then
	\[
	\Ln \mathrm{E}[X_n]=f(c),
	\]
	and 
	\[
	\Ln \pr\Big( G^d_n \hbox{ contains no cycles} \Big) = e^{-f(c)}= \exp\left(\frac{c}{2(d-2)!}\right) \sqrt{1-\frac{c}{2(d-2)!}}.
	\]
\end{corollary}

\begin{lemma}\label{lem:hyp_no_complex_cycles}
	Let $p\sim c/n^{d-1}$ with $0<c<(d-2)!$. Let $Z_n$ be the random
	variable equal to the number of  cycles in $G^d_n$ that belong to connected components
	that are not unicycles. Then
	\[
	\Ln \mathrm{E}[Z_n]=0.
	\]	
\end{lemma}

Let $\mathcal{U}$ be the family of unlabeled $d$-hypergraphs
whose connected components are unicyclic. 

\begin{lemma}\label{lem:hyp_unicycles_prob}
	Let $p\sim c/n^{d-1}$ with $c>0$. Let $\Tcal\subset \mathcal{U}$ 
	be a finite set of unicycles. For each $H\in \Tcal$ let $X_{n,H}$ be the
	random variable that counts the connected components in $G^d_n$ that
	are isomorphic  to $H$, and set \[
	\lambda_H= \frac{\left(c e^{-c/(d-2)!}\right)^{|H|}}{aut(H)}.
	\]	
	Then  $X_{n,H}$ converges in distribution
	to a Poisson variable with mean $\lambda_H$ as $n\to\infty$ and
	the $X_{n,H}$ are asymptotically independent, that is
	\[
	\Ln \pr\left(
	\bigwedge_{H\in \Tcal} X_{n,H}=a_H
	\right) = \prod_{H\in T} 
	e^{-\lambda_H}\frac{\lambda_H^{a_H}}{a_H!}.
	\]
\end{lemma}

Define the fragment of hypergraph as the collection of components that are unicycles and let $\frag_n$ be the fragment of the random hypergraph $G_n^d$.   

\begin{lemma}\label{lem:hyp_finite_fragment}
	Let $p\sim c/n^{d-1}$ with $0<c<(d-2)!$. Then the limit 
	\[
	\Ln \mathrm{E}\big[ |\frag_n| \big]
	\]
	exists and is a finite quantity. 
\end{lemma}

\begin{lemma}\label{lem:hyp_zeroone_fragment}
	Let $p\sim c/n^{d-1}$ with $0<c<(d-2)!$. Let $\phi$ be a FO sentence and
	let $H\in \Ucal$.
	Then
	\[
	\Ln \pr \Big(
	G^d_n\models \phi \, \Big| \, \frag_n\simeq H 	
	\Big)= \text{ $0$ or $1$.}
	\]
	Moreover, the value of the limit depends only on $\phi$ and $H$, 
	and not on $c$. 
\end{lemma}

As for graphs, we divide the proof of Theorem \ref{th:hypergraph} into several cases. Along the way we analyze the distribution of the fragment and the number of automorphisms in hypergraphs. 

\subsection{No gap when $c \ge (d-2)!$}

The arguments here mirror exactly those in \Cref{sec:nogap_supercritical}.
For each $k$ let $X_{n,k}$ be the random variable equal to the number of  $k$-cycles in $G^d_n$.
Then
\[
X_{n,\leq k}= X_{n,2}+ \dots + X_{n,k}
\xrightarrow[n\to\infty]{\text{d}}
\hbox{Po}\left( \sum_{i=2}^k \frac{(c/(d-2)!)^k}{2k} \right).
\] 
If $c\geq (d-2)!$ then  $\sum_{i=2}^k \frac{(c/(d-2)!)^k}{2k}$ tends to infinity and 
we can use the Central Limit Theorem to approximate any  $p\in (0,1)$ 
with FO statements of the form ``$X_{n,\leq k} \leq a$''. 

\subsection{At least when gap for $c < c_0$}

Let $f(c)$ be  as in \Cref{eq:hyp_fc}, and let $c_0$ be the unique solution 
of $e^{-f(c)}=1/2$ lying in $[0,(d-2)!]$. Because of \Cref{lem:hyp_acyclic_prob}, $c_0$ 
is the only value for which
\[
\Ln \pr\big(G^d_n\big) = 1/2,
\]
where $p(n)\sim c/n^{d-1}$. One can check that this is achieved when
the expected degree $c/(d-2)!$ is $\simeq 0.898$, independently of $d$. \par
Let $p(n)\sim c/n^{d-2}$ with $0<c<c_0$. 
Because of \Cref{lem:hyp_zeroone_fragment}, for any FO sentence $\phi$
\[
\Ln \pr\big(
G^d_n\models \phi \, \Big| \text{ $G^d_n$ contains no cycles }
\big)= \text{ $0$ or $1$.}
\]
From this point we  continue as in \Cref{sec:c0} to show that 
$[1-e^{-f(c)}, e^{-f(c)}]$ is a gap of $\overline{L_c}$.

\subsection{Asymptotic distribution of the fragment for $c<(d-2)!$}

The same proof of \Cref{lem:prob-fragment} can be used to prove the following result.
\begin{theorem}\label{lem:hyp_prob_fragment}
	Let $p(n)\sim c/n^{d-1}$ with $0<c<(d-2)!$. Let $H\in \Ucal$.
	Then
	\[
	\Ln \pr \big(
	\frag_n \simeq H
	\big)	= e^{-f(c)} \frac{\left(e^{-c/(d-2)!}c\right)^{|H|}}{aut(H)}.
	\]
\end{theorem}

For each $H\in \Ucal$  define 
$p_H(c)=p_H=\Ln \pr\Big(
\frag_n\simeq H\Big)
$.
Consider the set
\[
S_c=\left\{
\sum_{H \in \Tcal} p_H(c) \, : \, \Tcal\subseteq \Ucal
\right\}.
\]

One can proceed exactly as in \Cref{th:partialsums} to prove the following:
\begin{theorem}
Let $0<c<(d-2)!$. Then $\overline{L_c}=S_c$. 
\end{theorem}

\subsection{A lower bound on the number of automorphisms of unicyclic hypergraphs}

Let $H$ be an hypergraph and $h\in E(H)$ an edge. We call a vertex $v$ lying in 
$e$ \emph{free} if $e$ is the only edge that contains $v$. We denote by $\mathrm{free}(h)$ the
number of free vertices in $e$. Notice that
\[
\aut(H)\geq \prod_{h\in E(H)} \mathrm{free}(h)!,
\]
because free vertices inside an edge can be permuted without restriction. 
Given a unicycle $H$ we define the \emph{leaves} of $H$ as the edges $e\in E(H)$
that  contain only one non-free vertex. 


\begin{lemma}\label{lem:hyp_aut_bound}
	Let $H\in \Ucal$ be a $d$-hypergraph. Then,
	\[ \frac{(d-2)!^{|H|}}{\aut(H)} \leq \frac{(d-2)^2}{(d-1)^2}.\]
\end{lemma}

\begin{proof}
	
	It suffices to prove the statement for unicycles, because
	\[ \frac{(d-2)!^{|H|}}{\aut(H)}\leq \prod_{i} \frac{(d-2)!^{|H_i|}}{\aut(H_i)}, \]
	where the $H_i$ are the connected components of $H$.\par
	Let $\lambda$ be the number of 
	leaves in $H$. We show by induction that
	\begin{equation}\label{ineqleaves}
	\prod_{h\in E(H)} \frac{(d-2)!}{\mathrm{free}(h)!} \leq 
	\left(\frac{d-2}{d-1}\right)^\lambda.
	\end{equation}
	If $\lambda=0$ then $H$ is a cycle and each  of
	its edges contains exactly $d-2$ free vertices, so that 
	\[	\prod_{h\in E(H)} \frac{(d-2)!}{\mathrm{free}(h)!}=1
	,\]
and $H$ satisfies \eqref{ineqleaves}.
	Now let $H$ be a unicycle satisfying 
	\eqref{ineqleaves}. Add a new edge $h^\prime$ to $H$ 
	to obtain another unicycle $H^\prime$. Since 
	$h^\prime$ intersects $H$ in only one vertex $v$, it follows that $h^\prime$ is a leaf of $H^\prime$. There are
	two possibilities:
	\begin{itemize}[leftmargin=*]
	\item $\lambda(H^\prime)=\lambda(H)$. In this
	case no new leaves are created with the addition of
	$h^\prime$. This means that $v$ is a free vertex in one
	leaf $g$ of $H$ (that is, $h^\prime$ ``grows" out of $g$), 
	and
	\[	\prod_{h\in E(H^\prime)} \frac{(d-2)!}{\mathrm{free}(h)!}=
	\prod_{h\in E(H)} \frac{(d-2)!}{\mathrm{free}(h)!}.\]
	\item $\lambda(H^\prime)=\lambda(H)+1$. In this case
	$h^\prime$ intersects an edge of $H$ that is not a leaf.
	The case that maximizes $\prod_{h\in E(H^\prime)} 
	\frac{(d-2)!}{\mathrm{free}(h)!}$ is   when $h^\prime$ grows out  
	of a free vertex of an edge in $H$ with exactly
	$d-2$ free vertices. In this case
	\[	\prod_{h\in E(H^\prime)} \frac{(d-2)!}{\mathrm{free}(h)!}=
	\frac{d-2}{d-1} \prod_{h\in E(H)} 
	\frac{(d-2)!}{\mathrm{free}(h)!},\]
	and $H^\prime$ satisfies \eqref{ineqleaves} as well.
	\end{itemize}
	Finally, as all unicycles can be obtained adding edges
	to a cycle successively, \eqref{ineqleaves} holds for all
	unicycles. \par
	
	To prove the original statement consider the cases 
	$\lambda=0$, $\lambda=1$ and $\lambda\geq 2$.
	\begin{itemize}[leftmargin=*]
		\item
		If 
		$\lambda=0$ then $H$ is a cycle of length $l\geq 2$ 
		and $\aut(H)=(d-2)!^l 2l$, yielding
		\[ \frac{(d-2)!^{|H|}}{\aut(H)}=\frac{1}{2l} 
		\leq \frac{(d-2)^2}{(d-1)^2},\]
		since $1/2l\leq 1/4 \leq (d-2)^2/(d-1)^2$ for all 
		$l\geq 2, d\geq 3$.
		\item
		If $\lambda=1$ then $H$ is a cycle with a path attached
		to it. In this case, $H$ has a non-trivial  automorphism (a reflection of the cycle) 
 and as a consequence 
		$2\prod_{h\in E(H)}\mathrm{free}(h)! \leq \aut(H)$. Using this
		and \eqref{ineqleaves} we get
		\[ \frac{(d-2)!^{|H|}}{\aut(H)}\leq \frac{1}{2}
		\prod_{h\in E(H)} \frac{(d-2)!}{\mathrm{free}(h)!}\leq
		\frac{1}{2} \left( \frac{d-2}{d-1}\right)\leq
		\left( \frac{d-2}{d-1}\right)^2,	
		\]
		as we wanted.
		\item
		Finally, when $\lambda \geq 2$ the relation  \eqref{ineqleaves}
		suffices, since
		\[
		\prod_{h\in E(H)} \frac{(d-2)!}{\mathrm{free}(h)!}\leq
		\left( \frac{d-2}{d-1}\right)^\lambda \leq
		\left( \frac{d-2}{d-1}\right)^2.	
		\]
		
	\end{itemize}
	
	\par

\end{proof} 
\subsection{Some families of unicycles}

In this section we introduce three families of hypergraphs having a small number of automorphisms, and that will be used in the subsequent proofs.  \par
\begin{itemize}[leftmargin=*]
	\item
	Let $T_{\alpha,\beta}$ denote the hypergraph consisting of
	a triangle (as a $d$-hypergraph) with two paths of length $\alpha$ and $\beta$ 
	respectively attached to two of its free vertices, 
	each one from a different edge. One can check that
	\[
	\frac{(d-2)!^{|T_{\alpha,\beta}|}}{\aut(T_{\alpha,\beta})}=
	\frac{(d-2)!^{\alpha+\beta+3}}{\aut(T_{\alpha,\beta})}=\begin{cases}
	\left(\frac{d-2}{d-1}\right)^2 \text{ for } \alpha\neq\beta \text{,   }\\
	\\
	\frac{1}{2}\left(\frac{d-2}{d-1}\right)^2\text{ otherwise.} 
	\end{cases}
	\]
	Let $\mathcal{T}$ be the family of hypergraphs $\{T_{\alpha,\beta} \colon\alpha, \beta>0\}$.
	Then for  $k\geq 4$
	\begin{equation} \label{eqn:triangles}
	\sum_{H\in \mathcal{T}, \,|H|=k} \frac{(d-2)!^{|H|}}{\aut(H)}=
	\sum_{\alpha=1}^{\lfloor{\frac{k-3}{2}}\rfloor}
	\frac{(d-2)!^{k}}{\aut(T_{\alpha,k-3-\alpha})}= 
	\frac{k-4}{2} \left( \frac{d-2}{d-1} \right)^2.
	\end{equation}
	\item
	Let $B_{\alpha,\beta}$ denote the hypergraph
	consisting of a two-cycle with two paths 
	of length $\alpha$ and $\beta$ 
	respectively attached to two of its free vertices, 
	each one from a different edge. In this case
	\[
	\frac{(d-2)!^{|B_{\alpha,\beta}|}}{\aut(B_{\alpha,\beta})}=
	\frac{(d-2)!^{\alpha+\beta+2}}{\aut(B_{\alpha,\beta})}=
	\begin{cases}
	\frac{1}{2}\left(\frac{d-2}{d-1}\right)^2 
	\text{ for } \alpha\neq\beta \text{,}\\
	\\
	\frac{1}{4}\left(\frac{d-2}{d-1}\right)^2\text{ otherwise.} 
	\end{cases}
	\]	
	Let $\mathcal{B}= \{ B_{\alpha,\beta} \colon \alpha, \beta>0\}$.
	Then for  $k\geq 3$ 
	\begin{equation}  \label{eqn:twocycles}
	\sum_{H\in \mathcal{B}, |H|=k} \frac{(d-2)!^{|H|}}{aut(H)}=
	\sum_{\alpha=1}^{\lfloor{\frac{k-2}{2}}\rfloor}
	\frac{(d-2)!^{k}}{\aut(B_{\alpha,k-2-\alpha})}= 
	\frac{k-3}{4}\left( \frac{d-2}{d-1} \right)^2.
	\end{equation}
	\item
	We denote by $O_{\alpha,\beta}$,
	the hypergraph formed by attaching a path of length
	$\beta$ to a free vertex of a cycle of
	length $\alpha$. One can check that 
	$|O_{\alpha,\beta}|=\alpha+\beta$ and
	that $\frac{(d-2)!^{\alpha+\beta}}
	{\aut(O_{\alpha,\beta})|}=
	\frac{1}{2}\left(\frac{d-2}{d-1}\right)$.
	
	Let  $\mathcal{O}= \{O_{\alpha,\beta}\colon \alpha>1, \beta>0\}$.
	Then for  $k\geq 2$
	\begin{equation}  \label{eqn:cycles}
	\sum_{H\in \mathcal{O}, |H|=k} \frac{(d-2)!^{|H|}}{\aut(H)}=
	\sum_{\alpha=2}^{k-1}
	\frac{(d-2)!^{k}}{\aut(O_{\alpha,k-\alpha})}= 
	\frac{k-2}{2}\left( \frac{d-2}{d-1} \right).
	\end{equation}
\end{itemize}

\subsection{$\overline{L_c}$ is always a finite union of intervals}
Fix $0<c<(d-2)!$. Let $H_1,\dots, H_n,\dots$ be an enumeration of $\Ucal$
such that $p_{H_i}\leq p_{H_j}$ for all $i\leq j$. As before we shorten
$p_{H_i}$ to $p_i$. Analogously to \ref{sec:finitegaps} we need
to prove that for  $i$ large enough 
\[
p_i \leq \sum_{j>i} p_j.
\] 
Let $f=f(c)$ be as defined in \Cref{eq:hyp_fc}, and let 
$s=\frac{c}{(d-2)!} e^{-c/(d-2)!}$. Because of \Cref{lem:hyp_prob_fragment}
we have that
\begin{equation} \label{eq:hyp_prob_fragment}
p_i= e^{-f}  s^{|H_i|}  \frac{(d-2)!^{|H_i|}}{aut(H_i)}.
\end{equation}
For $i>0$ we define $k(i)$ as the unique integer such that
\[
e^{-f}  s^{k(i)-1} \left(\frac{d-2}{d-1}
\right)^2 
\geq p_i  > 
e^{-f}  s^{k(i)} \left(\frac{d-2}{d-1}
\right)^2 
\]
Notice that because of  Lemma, \eqref{lem:hyp_aut_bound}, 
we have $|H_i|\leq k(i)-1$. 

As a  consequence, if $k=k(i)\geq 4$ then  
\begin{align}
\nonumber
\sum_{j>i} p_i &\geq 
s^k \sum_{H\in U_k}\frac{(d-2)!^k}{\aut(H)} 
\\&
\nonumber
\geq s^k\frac{k-4}{2} \left( \frac{d-2}{d-1} \right)^2.
\end{align}
This is obtained taking into account only the hypergraphs in $\mathcal{T}$ and 
using \Cref{eqn:triangles}.
The last inequality implies that if $k(i)$ is such that 
$\frac{1}{s} \leq \frac{k(i)-4}{2}$ then $p_i\leq \sum_{j>i} p_j$. This
clearly holds for $i$ large enough, hence $\overline{L_c}$ is a finite
union of intervals, as needed to be proved.

\subsection{No gap when  $c_0 \leq c < (d-2)!$}

Fix $c_0 \leq c < (d-2)!$. Let $H_1,\dots, H_n, \dots$ 
be an enumeration of $\Ucal$ satisfying the same conditions as before. 
We want to show that for all $i$ 
\begin{equation} \label{eqn:kakeya}
p_i \leq \sum_{j>i} p_j.
\end{equation} 
\par
Notice that  $s=\frac{c}{(d-2)!} e^{-c/(d-2)!}$
satisfies that
\[
\frac{1}{3} < s < \frac{1}{e},
\]
because $0.898\leq c/(d-2)! < 1$.
The following inequalities are obtained using
\Cref{eqn:triangles,eqn:twocycles,eqn:cycles} respectively, together with the formula
for the sum of an arithmetic-geometric series and the 
fact that $1/3 < s$.

\begin{equation} \label{eq:hyp_total_triangles}
\sum_{H \in \mathcal{T}, |H|\geq k} p_H \geq e^{-f} s^k 
\frac{6k-21}{8}  
\left( \frac{d-1}{d-2}\right)^2 \quad \text{ for  $k\geq 4$.}
\end{equation}

\begin{equation} \label{eq:hyp_total_twocycles}
\sum_{H \in \mathcal{B}, |H|\geq k} p_H \geq e^{-f} s^k 
\frac{6k-15}{16}  
\left( \frac{d-1}{d-2}\right)^2\quad \text{ for  $k\geq 3$.}
\end{equation}

\begin{equation} \label{eq:hyp_total_cycles}
\sum_{H \in \mathcal{O}, |H|\geq k} p_H \geq e^{-f} s^k 
\frac{6k-9}{8}  
\left( \frac{d-1}{d-2}\right)\quad \text{ for  $k\geq 2$.}
\end{equation}

Assume first that $k=k(i)\geq 5$. Then 
\begin{align*}
\sum_{j>i} p_j &\geq
e^{-f} s^k \left[ \frac{18k-57}{16}
\left(\frac{d-2}{d-1}\right)^2
+
\frac{6k-9}{8} \left(\frac{d-2}{d-1}\right)
\right]\\
&\geq
e^{-f} s^k  \frac{30k-75}{16}
\left(\frac{d-2}{d-1}\right)^2
\geq e^{-f} s^k  3  \left(\frac{d-2}{d-1}\right)^2\\
&
\geq e^{-f} s^{k-1}  \left(\frac{d-2}{d-1}\right)^2 \geq p_i,
\end{align*}
as  was to be proven. \par
Otherwise, suppose that $k=k(i)\leq 4$.
Notice that because of \Cref{lem:hyp_aut_bound} necessarily $|H_i|\leq 3$.
We have three cases:
\begin{itemize}[leftmargin=*]
\item $|H_i|=3$. In this case, the following enumeration of all (unlabeled)
unicycles of size $3$ gives that 
\[
e^{-f} s^3  \frac{1}{2}  \left(\frac{d-1}{d-2}\right) 
\geq p_i.
\]
\begin{figure}[H]
	\begin{subfigure}{.3\linewidth}
		\centering
		\includegraphics[width=0.7\linewidth]{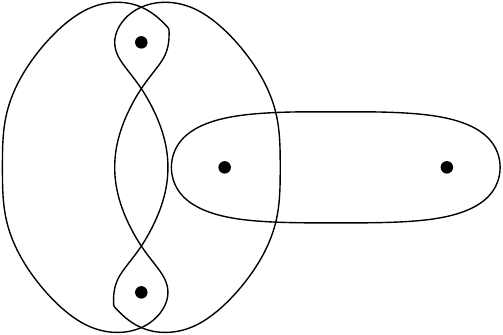}
		\caption{$\frac{(d-2)!^3}{\aut(H)}=\frac{1}{2}
			\left(\frac{d-1}{d-2}\right)$.}
	\end{subfigure}
	\hfill
	\begin{subfigure}{.3\linewidth}
		\centering
		\includegraphics[width=0.8\linewidth]{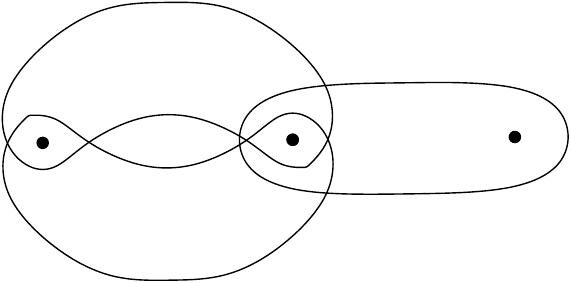}
		\caption{$\frac{(d-2)!^3}{\aut(H)}=\frac{1}{2}
			\left(\frac{1}{d-2}\right)$.}
	\end{subfigure}
	\hfill
	\begin{subfigure}{.3\linewidth}
		\centering
		\includegraphics[width=0.6\linewidth]{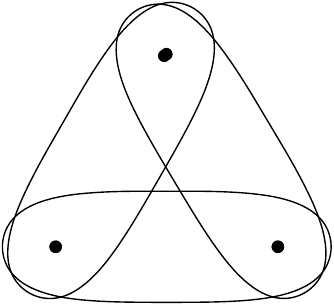}
		\caption{$\frac{(d-2)!^3}{\aut(H)}=\frac{1}{6}
			$.}
	\end{subfigure}
\end{figure}
Proceeding as before we obtain
\begin{align*}
\sum_{j>i} p_j &\geq
e^{-f} s^4 \left[ \frac{18 4-57}{16}
\left(\frac{d-2}{d-1}\right)^2
+
\frac{6 4-9}{8} \left(\frac{d-2}{d-1}\right)
\right]\\
&\geq
e^{-f} s^4 \left[ \frac{15}{16} \frac{1}{2} 
\left(\frac{d-2}{d-1}\right)
+
\frac{30}{8}  \frac{1}{2} \left(\frac{d-2}{d-1}\right)
\right]\\
&
\geq e^{-f} s^{4}  \frac{3}{2} \left(\frac{d-2}{d-1}\right) \geq 
e^{-f} s^{3}  \frac{1}{2} \left(\frac{d-2}{d-1}\right) \geq
p_i.
\end{align*}

\item $|H_i|=2$. In this case $H_i$ is the $2$-cycle, and 
$p_i=e^{-f} s^2  \frac{1}{4}$. Using \Cref{eq:hyp_total_twocycles,eq:hyp_total_cycles} we obtain
\begin{align*}
\sum_{j > i} p_j &\geq p_{C_3} + 
\sum_{H\in \mathcal{B}} p_H + 
\sum_{H\in \mathcal{O}} p_H\\ &\geq
e^{-f} s^3  \left[
\frac{1}{6} + 
\frac{3}{16}\left(\frac{d-2}{d-1}\right)^2
+ \frac{9}{8} \left(\frac{d-2}{d-1}\right)
\right]\\
&\geq
e^{-f} s^3  \left[
\frac{4}{6}\frac{1}{4} + 
\frac{3}{16}\frac{1}{4}
+ \frac{18}{8}  \frac{1}{4}\right]
\geq e^{-f} s^3  3  \frac{1}{4}
\geq p_i.
\end{align*}
 
\item $|H_i|=0$. In this case $H_i$ is the empty graph
and $p_i\geq 1/2$ by hypothesis. 

\end{itemize}

\section{Concluding remarks}

It can be shown that in Theorem \ref{th:main} the number of intervals in which $ \overline{L}_c$ decomposes when $c<c_0$ is unbounded as $c \to 0$. It would be interesting to determine at which rate the number of intervals grows as $c \to 0$. It is a delicate issue, since the decreasing ordering of the possible fragments according to their probabilities changes with $c$ in a  complicated way. 


\bibliographystyle{abbrvnat}
\bibliography{biblio}

%
%
%
%
%
%
%
%

\end{document}